\documentclass{amsart}
\usepackage{latexsym,amsxtra,amscd,ifthen}
\usepackage{amsfonts}
\usepackage{verbatim}
\usepackage{amsmath}
\usepackage{amsthm}
\usepackage{amssymb}

\theoremstyle{plain}

\newtheorem*{prop}{Proposition}

\newtheorem*{thrm}{Theorem}
\newtheorem*{cor}{Corollary}
\newtheorem*{corollary}{Corollary}

\newtheorem*{lemma}{Lemma}

 \theoremstyle{remark}
 
\newtheorem*{ex}{Example}

\newtheorem*{defn}{Definition}

\newtheorem*{rem}{Remark}

\newtheorem*{rems}{Remarks}

\theoremstyle{definition}

 \DeclareMathOperator{\id}{id}

 \DeclareMathOperator{\ad}{ad}
  
\DeclareMathOperator{\Alg}{Alg}

\begin{document}

\title[Ambiskew Hopf algebras]{Ambiskew Hopf algebras }

\author{K.A. Brown and M. Macauley}

\address{Brown: Department of Mathematics,
University of Glasgow, Glasgow G12 8QW, UK}

\email{Ken.Brown@glasgow.ac.uk}

\maketitle
\begin{abstract} Necessary and sufficient conditions are obtained for
an ambiskew polynomial algebra $A$ over a Hopf $k$-algebra $R$ to
possess the structure of a Hopf algebra extending that of $R$, in
which the added variables $X_+$ and $X_-$ are skew primitive. The
coradical filtration is calculated, many examples are
described, and properties determined.
\end{abstract}

\section{Introduction} \subsection{Ambiskew algebras} \label{ambiskew}
The ingredients to define an \emph{ambiskew polynomial algebra}
$A$ are a field $k$, arbitrary except where otherwise stated; a
$k$-algebra $R$, not necessarily commutative; a $k$-algebra
automorphism $\sigma$ of $R$; a central element $h$ of $R$; and a
nonzero element $\xi$ of $k$. Define $A = A(R, X_+, X_-, \sigma,
h, \xi)$ to be the $k$-algebra generated by $R$ and indeterminates
$X_+$ and $X_-$, subject to the relations
\begin{align} X_{+}r &= \sigma(r)X_{+} \,; \label{plus} \\   X_{-}r &=  \sigma^{-1}(r)X_{-} \; ; \label{plusplus}
 \\ X_+X_- &= h + \xi X_-X_+ \;, \label{minus} \end{align}
for all $r \in R$. These algebras were given this name in
\cite{Jor00} (for the case of commutative $R$), although they had appeared previously in the
literature - for example, in \cite{JoW96}. As explained in \cite{Jor00}, they are closely
related to down-up algebras \cite{BeR98} and to generalised Weyl
algebras \cite{Bav92}.

Let $R$ and $A$ be as above, and suppose that $R$ is a Hopf
$k$-algebra. Our main theorem gives necessary and sufficient
conditions for $A$ to be a Hopf algebra with $R$ as Hopf
subalgebra and such that $X_+$ and $X_-$ are \emph{skew primitive}
- that is,
\begin{equation} \label{deltaxpm1} \Delta(X_{\pm}) = X_{\pm}
\otimes r_{\pm} + l_{\pm} \otimes X_{\pm},\end{equation} for some
$l_{\pm}, r_{\pm} \in R$. Before stating the result we need to fix
some more notation.

\subsection{Notation} \label{notation} Our standard reference for Hopf algebras will be
\cite{Mo}. Let $H$ be a Hopf $k$-algebra. The comultiplication, counit and antipode maps are denoted
$\Delta$, $\varepsilon$ and $S$, respectively. Unadorned tensor products
are understood to be over $k$. For $h \in H$, we
write the coproduct $\Delta (h)$ as $\sum h_1 \otimes h_2.$ We
shall assume always that $S$ is bijective. The group of grouplike
elements of $H$ is denoted $G(H)$, and, for $g,w \in G(H)$ we
write $P_{g,w}(H)$ for the $g,w$-primitive elements of $H$, that
is the vector space of elements $h$ of $H$ for which $\Delta (h) =
h \otimes g + w \otimes h.$ Given a character $\chi$ of $H$, (that
is, a $k$-algebra homomorphism from $H$ to $k$), the \emph{right
winding automorphism} $\tau^r_{\chi}$ of $H$ is defined by
$\tau^r_{\chi}(h) = \sum h_1 \chi(h_2)$ for $h \in H$; and
$\tau^l_{\chi} $ denotes the corresponding \emph{left} winding
automorphism, \cite[I.9.25]{BG2}. The \emph{left adjoint action} of $y
\in H$ on $H$ is given by $\mathrm{ad}_l (y)(h) = \sum y_1 h
S(y_2)$ for $h \in H.$ The centre of a ring or group $X$ is
denoted by $Z(X)$.

\subsection{Main theorem} \label{main} First, one notes (Lemma \ref{relabel}) that,
if $A$ is an ambiskew Hopf $k$-algebra satisfying
\eqref{deltaxpm1}, then one can change the ambiskew variables in
$A$ to new ones, which we continue to denote $X_{\pm}$, for which
\begin{equation} \label{hatform2} \Delta (X_{\pm}) =
X_{\pm} \otimes 1 + y_{\pm} \otimes X_{\pm},\end{equation} for
$y_+,y_- \in G(R).$ Using this choice of variables makes the
statement of the theorem somewhat easier:
\begin{thrm} The ingredients which are necessary and sufficient to
construct an ambiskew Hopf algebra $A = A(R,X_+,X_-,\sigma,h,\xi)$
over the Hopf algebra $R$, with ambiskew variables $X_{\pm}$
satisfying \eqref{hatform2}, are as follows:
\begin{itemize} \item a character $\chi$ of $R$;
\item an element $z \in Z(R) \cap G(R)$; \item an element $h \in
P_{1,z}(R) \cap Z(R)$; \item elements $y_+, y_- \in Z(G(R))$ with
$z = y_+ y_-$ and $\chi (y_+) = \chi (y_-) =: \xi$; \item defining
$\sigma := \tau^l_{\chi}$, it is required that
$$ \sigma \quad = \quad \mathrm{ad}_l (y_+) \circ \tau^r_{\chi}.$$
\end{itemize}
\end{thrm}

This is Theorem \ref{newhopf}, with the rest
of $\S 2$ containing the proof. The case where $R$ is affine
commutative over the algebraically closed field $k$ was obtained
previously as the main result of \cite{Har08}. Even in the
commutative case, however, we believe that the formulation given
here makes it easier to determine the possible ambiskew extensions
of a given Hopf algebra $R$.

\subsection{Consequences and comments} \label{comments}
(i) \textbf{Special cases.} The specialisations of the main
theorem to the cases where $R$ is commutative or cocommutative are
stated respectively as Propositions 3.2 and 3.3. As
mentioned above, Proposition 3.2 is a result of Hartwig.

(ii) \textbf{Examples.} As shown in Proposition 3.1, it is an easy
consequence of the main theorem that, when $A$ is an ambiskew Hopf
algebra as above, then \emph{either $\xi = \pm 1$ or $h$ is a
scalar multiple of $z - 1$}. This considerably helps the
cataloguing of ambiskew Hopf extensions.

In \eqref{poly} all the possibilities for $A$ are listed when $R =
k[t]$, the polynomial algebra in a single variable. One gets the
enveloping algebra of $\mathfrak{sl}(2, k)$, the enveloping
algebra of the 3-dimensional Heisenberg Lie algebra, and a
1-parameter family of enveloping algebras of 3-dimensional
solvable Lie algebras.

In \eqref{Laurent}, the case $R = k[t^{\pm 1}]$ is described: the
basic example here is the quantised enveloping algebra
 $A = U_q (\mathfrak{sl}(2,k)).$ However there are a number of
 variants and degenerate cases to be considered.

In \eqref{qsl} we consider the case $R = U_q
(\mathfrak{sl}(2,k)).$ Examples occur where
$y_{\pm}$ are not central in $R$, and interesting, apparently new,
Hopf algebras are obtained as ambiskew extensions of $U_q
(\mathfrak{sl}(2,k)).$

(iii) \textbf{Coradical filtration.} For $R$ and $A$ as in the
main theorem, but assuming now that $k$ has characteristic 0, the
coradical filtration of $A$ is described in Theorem
\ref{coradthrm}. As special cases we recover the results for $A =
U_q (\mathfrak{sl}(2))$ of \cite{Mo2} and \cite{B}.

(iv) \textbf{Properties of ambiskew Hopf algebras.} In $\S 6$ a number of properties of ambiskew algebras are examined - ring-theoretic properties in Theorem \ref{ring}, homological properties in Theorem \ref{homol} and Hopf-theoretic properties in Theorem \ref{hoppy}. The properties covered include primeness and semiprimeness, and when the algebra is PI, AS-Gorenstein and AS-regular. Values for the PI-degree, the GK-dimension and the global and injective dimensions are given.

\section{Hopf algebra structure}

\subsection{Iterated skew structure}\label{iterate}

Recall that, given a $k$-algebra $T$ and algebra automorphisms
$\alpha$ and $\beta$ of $T$, an $(\alpha,\beta)$-derivation
$\delta$ of $T$ is a $k$-linear endomorphism of $T$ such that
$\delta(ab) = \alpha(a)\delta(b) + \delta(a)\beta(b)$ for all $a,b
\in T$.

Suppose here and throughout $\S$2 that $R$ and $A$ are as in
(\ref{ambiskew}). Recall the definition of a skew polynomial
algebra from \cite[1.2.1]{MR}. Then $A$ has the structure of an
iterated skew polynomial ring over $R$, as follows. The subalgebra
of $A$ generated by $R$ and $X_+$ is just the skew polynomial ring
$R[X_+; \sigma]$. We can extend $\sigma$ to an automorphism of
$R[X_{+}; \sigma]$, which we denote by the same symbol, by
defining $\sigma(X_+) = {\xi}X_+$. Define the $(
\sigma^{-1},\id)$-derivation $\delta \colon R[X_+;\sigma]
\rightarrow R[X_+; \sigma]$ by $\delta(R) = 0$, $\delta(X_+) =
-\xi^{-1}h$. One checks routinely that this is well-defined. Then
\[ A \cong R[X_+; \sigma][X_-; \sigma^{-1}, \delta].\]
Consequently, $A$ is a free left and right $R$-module, with basis
\begin{equation} \label{abasis} \{ {X_+}^m {X_-}^n \colon m, n \geq
0\}.\end{equation} In addition, $A \otimes A$ is a free left and
right $R \otimes R$-module, with basis \begin{equation}
\label{aabasis} \{ {X_-}^m {X_+}^n \otimes {X_-}^p{X_+}^q \colon
m,n,p,q \geq 0\}.
\end{equation}

\subsection{The key lemma.}
\begin{lemma} \label{change} Let $R$ and $A$ be as in (\ref{ambiskew}). \begin{enumerate} \item \label{bi} Suppose $R$ is equipped with a bialgebra structure, which
 extends to a bialgebra structure on $A:=(A,m,u,\Delta,\varepsilon)$ such
that \eqref{deltaxpm1} holds. Then
\begin{enumerate}
\item $\Delta(r_{\pm}) = r_{\pm} \otimes r_{\pm}$ and
$\Delta(l_{\pm}) = l_{\pm} \otimes l_{\pm}$; \item  $r_+r_- =
r_-r_+$ and $l_+l_-=l_-l_+ ;$ \item $\Delta(h) = h \otimes r_+r_-
+ l_+l_- \otimes h.$
\end{enumerate} \item \label{ho} Suppose $R$ is equipped with a Hopf algebra structure, which
 extends to a Hopf algebra structure on $A:=(A,m,u,\Delta,\varepsilon,S)$ such
that \eqref{deltaxpm1} holds. Define the character $\chi :=
\varepsilon \circ \sigma \colon R \rightarrow k$. Then
\begin{enumerate}
\item \label{group} $r_{\pm}$ and $l_{\pm}$ are grouplike, and
$r_+r_-$ and $l_+l_-$ are in $Z(R)$.

\item\label{windup} $\sigma = \ad_l(l_+) \circ \tau^r_{\chi} =
\ad_l(r_{+}) \circ \tau^l_{\chi}$.

\item \label{windup2} $\sigma = \ad_r(l_-) \circ \tau^r_{\chi} =
\ad_r(r_{-}) \circ \tau^l_{\chi}$.

\item \label{cent} $B := \langle r_+, r_-, l_+,l_- \rangle $ is an
abelian subgroup of $G(R)$.

\item \label{sigho} For all $g \in B$, $\sigma (g) = \chi(g)g$.

\item $\xi =  \chi (l_- r_+ ) = \chi (l_+ r_-).$
\end{enumerate}
\end{enumerate}
\end{lemma}

\begin{proof}

(i) Suppose that $A$ has a bialgebra structure extending that on
$R$. (a) The counit condition for $A$ implies that
\begin{align*} m( \id \otimes \varepsilon) \Delta(X_{\pm}) &= X_{\pm}
\varepsilon(r_{\pm}) + l_{\pm} \varepsilon (X_{\pm}) = X_{\pm}, \\
m(\varepsilon \otimes \id) \Delta(X_{\pm}) &=
\varepsilon(X_{\pm})r_{\pm} + \varepsilon(l_{\pm})X_{\pm} = X_{\pm}.
\end{align*} Rearranging these gives $(1 -
\varepsilon(r_{\pm}))X_{\pm} = l_{\pm} \varepsilon(X_{\pm})$ and
$(1 - \varepsilon(l_{\pm}))X_{\pm} = \varepsilon(X_{\pm}) r_{\pm}
.$ Since $A$ is a free $R$-module with basis
\eqref{abasis}, both sides must be zero, so $\varepsilon$ satisfies
the counit condition if and only if
\begin{equation} \label{counitiff} \varepsilon(X_{\pm}) = 0 \enspace
\textrm{and} \enspace \varepsilon(l_{\pm}) = \varepsilon(r_{\pm})
= 1. \end{equation} The coassociativity condition for $\Delta$
yields
\[(\id \otimes \Delta) \Delta(X_{\pm}) = (\Delta \otimes \id)
\Delta (X_{\pm});\] that is, \[X_{\pm} \otimes (\Delta(r_{\pm}) -
r_{\pm} \otimes r_{\pm}) = (\Delta(l_{\pm}) - l_{\pm} \otimes
l_{\pm}) \otimes X_{\pm}.\] Since $A \otimes A$ is a free $R
\otimes R$-module on basis \eqref{aabasis}, and $\Delta(R)
\subseteq R \otimes R$, both sides must be zero, so that the
coassociativity of $\Delta$ implies (i)(a).

   Using the fact that $\Delta$ preserves \eqref{minus} combined
with the freeness of $A \otimes A$ as an $R \otimes R$-module,
\eqref{aabasis}, one deduces in a similar way to (a) that (i)(b)
and (i)(c) hold, and also that
\begin{align} \sigma(l_-) \otimes r_+ &= \xi l_- \otimes \sigma^{-1}(r_+),
\label{deltaminusiff3}\\ l_+ \otimes \sigma(r_-) &= \xi
\sigma^{-1}(l_+) \otimes r_-. \label{deltaminusiff4}
\end{align}

(ii) (a) Note first that if $R$ is a Hopf algebra, the antipode
condition combined with (i)(a) implies that $r_{\pm}$ and
$l_{\pm}$ are invertible. So the first part of (a) follows from
(i)(a). We postpone the proof of the last part of (a).

(b) For all $a \in R$, using \eqref{deltaxpm1} and \eqref{plus},
\begin{align} \Delta(X_{+}) \Delta(a) &= (l_{+} \otimes X_{+}
+ X_{+} \otimes r_{+}) \left( \sum a_1 \otimes a_2 \right) \nonumber \\
&= \sum l_{+} a_1 \otimes X_{+} a_2 + X_{+}a_1 \otimes
r_{+}a_2 \nonumber \\ &= \sum l_{+} a_1 \otimes
\sigma(a_2) X_{+} + \sigma(a_1) X_{+} \otimes
r_{+}a_2. \label{deltacalc} \end{align}
Now $\Delta$ preserves
\eqref{plus} if and only if, for all $a \in R$,
\begin{align*} \Delta(X_{+}) \Delta(a) &= \Delta(\sigma(a))
\Delta(X_{+}).\end{align*} By \eqref{deltaxpm1} and
\eqref{deltacalc}, this is equivalent to
\begin{align*} \sum l_{+} a_1 & \otimes
\sigma(a_2) X_{+} + \sigma(a_1) X_{+} \otimes
r_{+}a_2 = \Delta(\sigma(a))(X_{+} \otimes r_{+} +
l_{+} \otimes X_{+});
\end{align*} that is,
\begin{align*}& \qquad \left( \sum l_{+} a_1 \otimes \sigma(a_2) -
\Delta(\sigma(a))(l_{+} \otimes 1) \right)(1 \otimes
X_{+})\\ & \qquad \qquad \qquad = \left( \sum
\Delta(\sigma(a))(1 \otimes r_{+}) - \sigma(a_1)
\otimes r_{+}a_2 \right) (X_{+} \otimes 1). \end{align*}
Therefore, $\Delta$ preserves \eqref{plus} if and only if, for all
$a \in R$,
\begin{align*}
 &
(\Delta(\sigma(a))(l_{+} \otimes 1) - (l_{+} \otimes
1)(\id \otimes \sigma)\Delta(a)) (1 \otimes
X_{+})\\
& \quad \qquad \qquad + (\Delta(\sigma(a))(1 \otimes r_{+})
- (1 \otimes r_{+})(\sigma \otimes \id)\Delta(a)) (X_{+}
\otimes 1) = 0.
\end{align*}
Using the linear independence of the $R \otimes R$-basis
\eqref{aabasis} of $A \otimes A$, this is equivalent to the
coefficient of each of the above two terms being 0. Since $r_{+}$
and $l_{+}$ are invertible, it follows that $\Delta$ preserves
\eqref{plus} if and only if  \begin{equation}
\label{deltaplusiff1} \Delta(\sigma(a)) = (\ad_l(l_{+}) \otimes
\id)(\id \otimes \sigma)\Delta(a),\end{equation} and
\begin{equation} \label{deltaplusiff2} \Delta(\sigma(a)) = (\id
\otimes \ad_l(r_{+}))(\sigma \otimes \id)\Delta(a).
\end{equation}

Applying $m(\id \otimes \varepsilon)$ [resp. $m(\varepsilon
\otimes \id)$] to both sides of \eqref{deltaplusiff1} [resp.
\eqref{deltaplusiff2}] yields (ii)(b).

(c) This is proved in the same way as (b), starting from
\eqref{plusplus} rather than \eqref{plus}.

(a) The last part of (a) now follows by comparing the two
expressions for $\sigma$ in terms of $\tau^l_{\chi}$ [resp.
$\tau^r_{\chi}$] obtained in (b) and (c).

(d),(e) Let $g \in S$ and let $\mu := \chi(g) \in k^*$. Now $g \in G(R)$ by (ii)(a), so that (ii)(b) and (c) yield
$$ \sigma(g) = \mu \ad_l(l_+)(g) = \mu \ad_l(r_+)(g) = \mu
\ad_r(l_-)(g) = \mu \ad_r(r_-)(g).$$ Taking $g \in \{l_{\pm},
r_{\pm} \}$ now implies (d), and then (e) follows at once.

(f) This follows by applying (e) to \eqref{deltaminusiff3} and
\eqref{deltaminusiff4}. \qedhere
\end{proof}

\subsection{Relabelling the variables} \label{relabel}Let $R$ be a Hopf algebra whose structure extends to a
Hopf algebra structure on $A$ such that \eqref{deltaxpm1} holds. Then $r_{+}$ and $r_{-}$ are invertible, by Lemma \ref{change}
(ii)(a). We can equally well think of $A$ as the ambiskew polynomial algebra with the variables $X_{\pm}$ replaced by
$X_{\pm}{r_{\pm}}^{-1}$, with slight adjustments to $\sigma$, $\xi$ and $h$, as in the following lemma. This has the advantage that
the new variables are $(1,y_{\pm})-$primitive, for suitable elements $y_{\pm}$ of $Z(G(R))$, thus significantly simplifying some
later calculations.

\begin{lemma} Suppose that $R$ and $A$ are as in (\ref{ambiskew}). Suppose that
 $R$ is equipped with a Hopf algebra structure, which
 extends to a Hopf algebra structure on $A:=(A,m,u,\Delta,\varepsilon,S)$ such
that \eqref{deltaxpm1} holds. Set $\chi := \varepsilon \circ
\sigma \colon R \rightarrow k$. \begin{enumerate} \item Then there
is an alternative ambiskew presentation of $A$,
\[ A = A(R, X_+{r_+}^{-1}, X_-{r_-}^{-1}, \widehat{\sigma}, \widehat{\xi}, \widehat{h}),\]
where \begin{itemize} \item $\widehat{\sigma} := \ad_r(r_+) \circ
\sigma = \ad_l(r_-)\circ \sigma$; \item $\widehat{\xi} := \xi
\chi(r_+r_-)^{-1}$; \item $\widehat{h} :=
\chi({r_+})^{-1}h(r_+r_-)^{-1}$. \end{itemize} \item Setting
$\widehat{X}_+ := X_+{r_+}^{-1}$, $\widehat{X}_- :=
X_-{r_-}^{-1}$, and $y_{\pm} := l_{\pm}r_{\pm}^{-1}$,
\begin{equation}\label{hatform} \Delta (\widehat{X}_{\pm}) =
\widehat{X}_{\pm} \otimes 1 + y_{\pm} \otimes
\widehat{X}_{\pm}.\end{equation}
\end{enumerate}
\end{lemma}

\begin{proof} (i) It is straightforward to check using Lemma
\ref{change} that the stated data affords an alternative
presentation of $A$ as an ambiskew algebra.

(ii) This is clear from \eqref{deltaxpm1}.
\end{proof}

\subsection{Main theorem}

\begin{thrm} \label{newhopf} Let $R$ be a Hopf $k$-algebra.
\begin{enumerate}
\item Fix the following data:
\begin{itemize}
\item a character $\chi$ of $R$; \item elements $z \in Z(R) \cap
G(R)$ and $h \in Z(R) \cap P_{1,z}(R);$ \item a factorisation $z = y_+ y_- = y_-y_+$ with $y_{\pm} \in Z(G(R))$ and
$\chi(y_+) = \chi (y_-) := \xi.$
\end{itemize}
Define $\sigma := \tau^{l}_{\chi}$ and assume that the above data
is chosen such that \begin{equation}\label{strange} \sigma =
\ad_{l}(y_+)\tau^{r}_{\chi}.\end{equation} Then the ambiskew
algebra $A = A(R,\widehat{X}_+,\widehat{X}_-,\sigma,h,\xi)$ has a
Hopf structure extending the given structure on $R$, and such that
(\ref{hatform}) holds. The antipode $S$ of $R$ extends to an
antipode of $A$, with \begin{equation}\label{aunt}
S(\widehat{X}_{\pm}) = -y_{\pm}\widehat{X}_{\pm}. \end{equation}

\item Conversely, suppose that $A$ is an ambiskew algebra
constructed from $R$ as in (\ref{ambiskew}), and that $A$ admits a
structure of Hopf algebra extending the Hopf structure of $R$, and
such that (\ref{deltaxpm1}) holds. Then, after a change of
variables as in Lemma (\ref{relabel}), $A$ can be presented with
ambiskew variables satisfying (\ref{hatform}), and such that the
conditions of (i) are satisfied.
\end{enumerate}
\end{thrm}

\begin{proof}(ii) Suppose that the ambiskew algebra
(\ref{ambiskew}) admits a structure of Hopf algebras extending
that of $R$, with $\Delta(X_{\pm})$ given by (\ref{deltaxpm1}).
Lemma (\ref{relabel}) shows that we can change the variables to
$\widehat{X}_{\pm} := X_{\pm}r_{\pm}^{\pm 1}$, and hence assume
that the ambiskew variables satisfy (\ref{hatform}), for some
elements $y_{\pm} \in G(R).$ For convenience, however, we retain
the notation $\sigma, h, \xi,$ after the change of variables.

Now let $\chi$ be the character $\varepsilon \circ \sigma$ of $R$ as in Lemma (\ref{change})(ii). By Lemma (\ref{change})(i)(b),(c)
and (ii)(a),(b), $\langle y_+, y_- \rangle$ is a central subgroup of $G(R)$, with
$$ z:= y_+y_- \in Z(R) \cap G(R), $$
and
$$ \Delta (h) = h \otimes 1 + z \otimes h. $$
By Lemma (\ref{change})(f),
$$ \chi (y_+) = \chi (y_-) = \xi. $$
Finally, (\ref{strange}) follows from Lemma (\ref{change})(ii)(b),
and the formulae for $S(\widehat{X}_{\pm})$ are determined by
(\ref{hatform}) and the fact that $y_{\pm}$ are group-like.

(i) Conversely, it is a routine (but fairly long) series of
calculations to check that, given the data as listed in (i), the
Hopf structure on $R$ can be extended to a Hopf structure on $A$
with $\Delta (\widehat{X}_{\pm})$ given by (\ref{hatform}), and
with (\ref{aunt}) holding.
\end{proof}

\section{Consequences of the Main Theorem}

\subsection{} We show first that once $\chi, z, y_+$ and $y_-$
have been chosen so as to satisfy the constraints of the Main
Theorem, there is little room for manoeuvre left in the choice of
$h$.

\begin{prop} \label{options} Let $R$ be a Hopf algebra and let $A = A(R,
\widehat{X}_+, \widehat{X}_-, \sigma, \xi,h)$ be an ambiskew Hopf
algebra with $R$ as a Hopf subalgebra and with $\widehat{X}_{\pm}$
satisfying (\ref{hatform}) for suitable $y_{\pm}$ in $G(R)$.
Recall that $h \in P_{1,z}(R)$ for $z := y_-y_+$, and $\xi :=
\chi(y_+) = \chi(y_-)$ by Theorem (\ref{newhopf}).

Then (at least) one of the following holds:
\begin{enumerate}
\item $\xi = \pm 1, \; \chi (h) = 0.$ \item $\xi = \pm 1,\; z =
1.$ \item $\xi \neq \pm 1, \; h = \frac{\chi(h)}{\xi^2 - 1}(z -
1).$
\end{enumerate}
\end{prop}

\begin{proof} From Theorem (\ref{newhopf}), $\Delta (h) = h
\otimes 1 + z \otimes h.$ Thus, calculating $\sigma (h)$ in two
ways using (\ref{strange}),
\begin{align*} \sigma(h) = \tau^l_{\chi}(h) &= \chi(h) + \xi^2h
\\
&= \mathrm{ad}_l(y_+)(\tau^r_{\chi}(h)) \\
&= \mathrm{ad}_l(y_+)\{ h + z \chi (h)\} \\
&= h + z\chi (h).
\end{align*}
Hence, $$(\xi^2 - 1)h \; = \; \chi (h) (z - 1), $$ and the
proposition follows.
\end{proof}

We'll give examples in $\S$4 to show that all three possibilities
occur.

\subsection{Special case: $R$ commutative.} \label{comm}
When Theorem (\ref{newhopf}) is specialised to the case where $R$ is commutative we obtain the main result of \cite{Har08}, albeit
stated somewhat differently from there:

\begin{prop} $\mathrm{(Hartwig)}$ Let $R$ be a commutative Hopf $k$-algebra and let
$\chi$ be a character of $R$ such that
\begin{equation}\label{centchar} \sigma \; := \; \tau^l_\chi \; = \; \tau^r_\chi.
\end{equation}
Choose the following: \begin{itemize} \item $y_{\pm} \in G(R)$
with $\xi := \chi (y_-) = \chi (y_+);$ \item $h \in P_{1,z}(R),$
where $z := y_-y_+$.
\end{itemize}
Then the ambiskew algebra $A := A(R,\widehat{X}_+,\widehat{X}_-,\sigma,\xi,h)$ has a Hopf structure extending that of $R$, with the
coproduct of $\widehat{X}_{\pm}$ given by (\ref{hatform}). Conversely, every ambiskew Hopf algebra extending the Hopf structure of
$R$ and satisfying (\ref{deltaxpm1}) arises in this way, after a change of ambiskew variables.
\end{prop}

Many interesting examples of ambiskew Hopf algebras with $R$
commutative are given in \cite[$\S 4$]{Har08}. See also \ref{poly}
and \ref{Laurent} below.

When $k$ has characteristic 0, so that $R$ is semiprime
\cite{W}, it is not hard to see that (\ref{centchar})
holds if and only if $\chi \in Z(G(R^{\circ})).$ (Here, $R^{\circ}$ denotes the Hopf dual of $R$.) We don't know if
the latter condition is sufficient in positive characteristic.

\subsection{Special case: $R$ cocommutative.} \label{cocomm} When
we are in the setting of the Main Theorem (\ref{newhopf}) and in addition $R$ is cocommutative, then $\tau^l_{\chi} = \tau^r_{\chi}$,
so that (\ref{strange}) forces $y_+ \in Z(R).$ Theorem (\ref{newhopf}) thus yields:

\begin{prop} Let $R$ be a cocommutative Hopf $k$-algebra. Choose
the following: \begin{itemize} \item a character $\chi$ of $R$;  \item $z,y_+, y_- \in G(R) \cap Z(R)$ with $z = y_+y_-$ and $\xi :=
\chi(y_+) = \chi(y_-);$
\item either (i) $h \in k(z - 1),$ or (ii) $h \in P(R)$. When (ii) holds, we also require $z = 1$, so that $y_- = y_+^{-1}$ and $\xi \in \{\pm 1 \}.$
\end{itemize}
Set $\sigma := \tau^l_{\chi}$. Then the ambiskew algebra $A := A(R,\widehat{X}_+,\widehat{X}_-,\sigma,\xi,h)$ has a Hopf structure
extending that of $R$, with the coproduct of $\widehat{X}_{\pm}$ given by (\ref{hatform}). Conversely, every ambiskew Hopf algebra
extending the Hopf structure of $R$ and satisfying (\ref{deltaxpm1}) arises in this way, after a change of ambiskew variables.
\end{prop}

Notice that the argument in the above proposition to deduce that $y_+$ (and therefore also $y_-$) are in $Z(R)$ applies in a similar
way to the general case, allowing us to conclude that, in the setting of Theorem (\ref{newhopf}),
\begin{equation*}
y_+ \textit{ and } y_- \newline \textit{ are in the centre of the
subalgebra of cocommutative elements of } R. \end{equation*}

\section{Examples} \label{examples}
\subsection{$R$ the polynomial algebra in a single variable.} \label{poly} Let $R = k[t]$, so $G(R) = \{1\}$ and $P(R) = kt.$ Propositions
(\ref{comm}) and (\ref{cocomm}) both apply. There are three cases, consisting of two singletons and a one-parameter infinite family.

(i) Take
$$ z = y_+ = y_- = \xi = 1, \; \; and \; \; h = t, $$
and $\chi \in \mathrm{char}(R)$ with $\chi (t) := \lambda \neq 0.$ Thus
$$ \sigma (t) = t + \lambda, $$
and
$$
  A \quad = \quad  A(k[t], X_+,X_-, \sigma,t,1)
\quad = \quad k \langle t, X_+,X_- \rangle,$$ with relations \begin{align*} X_+ t \quad&= \quad(t + \lambda )X_+,\\
X_- t \quad&= \quad(t - \lambda) X_-, \\
X_+ X_- - X_-X_+ \quad &=\quad t.
\end{align*}
After an obvious change of variables, $A \cong U(\mathfrak{sl}(2,k)).$

(ii) Fix the parameters as in (i), except that $\lambda = \chi (t) = 0.$ Thus $\sigma$ is the identity map and $A = k\langle t,
X_+,X_- \rangle$, with relations
$$ [X_+,t] = [X_-,t] = 0, \quad [X_+,X_-] = t,$$ so $A$ is the enveloping algebra of the 3-dimensional Heisenberg Lie algebra.

(iii) Fix the parameters as in (i) or (ii), except that $h = 0.$ Then $A =k\langle t, X_+,X_- \rangle$, with relations
$$ [t,X_+] = -\lambda X_+, \quad [t, X_-] = \lambda X_-, \quad [X_+,X_-] = 0.$$
So $A$ is the enveloping algebra of a 3-dimensional solvable Lie algebra with adjoint eigenvalues $\{\pm \lambda \}$, abelian when
$\lambda = 0.$

\subsection{$R$ the Laurent polynomial algebra in a single variable.} \label{Laurent} In this case the generic example is $U_q
(\mathfrak{sl}(2,k))$, but there are more variants than in (\ref{poly}). Let $R = k[t^{\pm 1}]$, so $ G(R) = \langle t \rangle.$
Propositions (\ref{comm}) and (\ref{cocomm}) apply. Start by fixing a character $\chi$ of $R$, with $\chi (t) := \eta \in k^*$, so that $\sigma (t) = \eta t.$ Choose also a group-like element
$$ z  =  t^m, \qquad (m \in \mathbb{Z}),$$ and a factorisation of $z$,
$$z \quad := \quad t^{\ell} t^n,$$
so that
$$ y_+ := t^{\ell}, \quad y_- := t^n; \quad (\ell,n \in \mathbb{Z}, \, m = \ell + n).$$
Since we must have $\chi (y_+) = \chi (y_-)$, we require
\begin{equation} \eta^{\ell - n} \quad = \quad 1. \label{root} \end{equation}

\textbf{Case (i):} Suppose that
\begin{equation} \ell \quad = \quad n. \label{quantum} \end{equation}
That is, $y_+ = y_- = t^{\ell}$ and $$m \quad = \quad 2\ell, $$ with \begin{equation} \chi (y_+) \quad = \quad \eta^{\ell} \quad =:
\quad \xi. \label{map} \end{equation} In the proposition below we initially assume that
\begin{equation}\label{nonzero} m \neq 0; \quad \textit{that is,} \quad z \neq 1. \end{equation}
Then Proposition (\ref{cocomm}) tells us that, for some $\lambda \in k$,
$$ h \quad = \quad \lambda (z - 1). $$
Again, we avoid degenerate cases initially by assuming that
\begin{equation} \label{ass} \lambda \quad \neq \quad 0. \end{equation}
Let $A$ be the ambiskew algebra with these parameters,
$$ A \quad = \quad A(k[t^{\pm 1}], X_+, X_-, \sigma, \lambda(t^m - 1), \eta^{\ell}).$$

\begin{prop} Fix the notation as at the start of the paragraph, and  assume (\ref{quantum}). \begin{enumerate}
\item Assume (\ref{nonzero}) and (\ref{ass}), and that $\eta$ is not an $\ell$th root of 1. Then $A$ is the Hopf algebra $A = k \langle t^{\pm 1},E,F \rangle$, with
\begin{align*} Et = \eta tE, \quad & \quad Ft = \eta^{-1}tF, \\
EF - FE &= \frac{t^{\ell} - t^{-\ell}}{\eta^{\ell/2} - \eta^{-\ell/2}}; \end{align*} moreover,
\begin{align*} \Delta (E) &= E \otimes 1 + t^{\ell} \otimes E, \\ \Delta (F) &= F \otimes t^{-\ell} + 1 \otimes F. \end{align*}
That is, $A$ is the variant of the standard quantised $\mathfrak{sl}(2)$ which is free of rank $\ell$ as a module over its subHopf
algebra $U_{q}(\mathfrak{sl}(2,k)),$ where $q = \eta^{\frac{\ell}{2}}.$
\item Assume (\ref{nonzero}) and (\ref{ass}), and that $\eta $ is an $\ell$th root of 1. Then $A$ is as in (i), except that
$$ EF -FE = t^{\ell} - t^{-\ell}. $$
\item Assume that $\lambda = 0$ so that $h = 0$. Allow arbitrary $m \in \mathbb{Z}$. Then $A$ is as in (i), except that $$ EF = FE. $$
Thus $A$ is a quantum affine 3-space with a line removed.
\end{enumerate}
\end{prop}

\begin{proof} (i) Make the change of variables $E := X_+, \; F := \lambda^{-1}(\eta^{\ell /2} - \eta^{-\ell /2})^{-1}X_-t^{-\ell}.$

(ii) Put $E := X_+, \; F:= \lambda^{-1}X_-t^{-\ell}.$

(iii) Put $E := X_+, \; F := X_-t^{-\ell}$. \qedhere
\end{proof}

\textbf{Case (ii):} Suppose that
\begin{equation} \ell \neq n, \label{noteq} \end{equation}
so that, by (\ref{root}),
\begin{equation} \eta \textit{ is a primitive } \omega \textit{th root of 1, where } \omega | (\ell - n). \label{arrow}
\end{equation}
Changing the ambiskew variables to $E := X_+, \, F := X_- t^{-n},$ we obtain:

\begin{prop} Fix the notation as at the start of the paragraph, and  assume (\ref{noteq}), so that (\ref{arrow}) also holds. We can
normalise so that either $\lambda = 0$ or $\lambda = 1$. Then $A = k \langle t^{\pm 1},E,F \rangle$ has relations
\begin{align*} Et = \eta tE; \quad &\, \quad Ft = \eta^{-1} tF, \\
EF - FE \quad &= \quad \lambda (t^{\ell} - t^{-n}),
\end{align*}
and coproduct given by
\begin{align*} \Delta (E) =& E \otimes 1 + t^{\ell} \otimes E, \\
\Delta (F) =& F \otimes t^{-n} + 1 \otimes F. \end{align*}
\end{prop}

\subsection{$R$ the quantised enveloping algebra of
$\mathfrak{sl}_2$}\label{qsl} We work with the standard presentation of $R$, namely
\begin{align*} k \langle E, F, K^{\pm 1} :& KE = q^2 EK \rangle \\
& KF = q^{-2} FK \\
& EF - FE = \frac{K - K^{ -1} }{q - q^{ -1} } \end{align*} where $q \in k$, $q \neq 0, \pm 1$. The coproduct is defined by requiring
$K \in G(R)$, $E \in P_{1,K}(R)$, $F \in P_{K^{-1},1}(R).$ There are two characters: (i) $\chi$ defined by $\chi(E) = \chi (F) = 0,
\chi (K) = -1,$ and (ii) the counit $\varepsilon.$

\textbf{Case (i):} \textbf{The non-trivial character} $\chi$. In this case
$$ \sigma = \tau^{\ell}_{\chi} : K \mapsto -K, \quad E \mapsto -E, \quad F \mapsto F; $$
whereas
$$ \tau^{r}_{\chi} : K \mapsto -K, \quad E \mapsto E, \quad F \mapsto -F.$$
To satisfy (\ref{strange}), we thus need
$$ \mathrm{ad}_{\ell}(y_+ ) : K \mapsto K, \quad E \mapsto -E, \quad F \mapsto -F. $$
For a group-like element $y_+ = K^r$ of $R$, this holds if and only if
$$ q^{2r} = q^{-2r} = -1. $$
Thus, for some positive integer $\widehat{\ell}$,
$$ q \textit{ is a primitive } 4\widehat{\ell}^{\textit{th}} \textit{ root of } 1, $$
with
$$ r = \widehat{\ell} s $$
for some odd positive integer $s$. In particular, \begin{equation} \label{centre1} Z(R) \cap G(R) = \langle K^{2\widehat{\ell}}
\rangle, \end{equation} so we can fix
$$ z := K^{2\widehat{\ell} m}, $$
for some $m \in \mathbb{Z}.$ Then
$$ y_- = zy_{+}^{-1}   = K^{2\widehat{\ell} m - \widehat{\ell} s} . $$
Therefore, noting that $s$ and $p := 2m - s$ are odd,
$$ \xi := \chi (y_+) = \chi (y_-) = (-1)^{\widehat{\ell}} .$$
Finally, choose $h \in Z(R) \cap P_{1,z}(R)$; that is,
$$ h = \lambda (1 - z) $$
for some $\lambda \in k$, which, after normalising, we can take to be either 0 or 1.

\textbf{Case (ii):} $\chi = \varepsilon$, \textbf{the counit.} Thus $\sigma$ is the identity map on $R$, and (\ref{strange}) forces
$y_+ \in Z(R) \cap G(R).$ Since  $z$ must always be central, we get $z, y_=, y_- \in Z(R) \cap G(R)$ with $z = y_+y_-$, and $\xi :=
\chi(y_+) = \chi(y_-) = 1.$

Define $r$ by
$$ r := \left\{\begin{array}{ll}  \ell & \textit{ if $q$ is a primitive $\ell^{th}$ root of 1, $\ell$ odd}\\
\ell/2 & \textit{ if $q$ is a primitive $\ell^{th}$ root of 1, $\ell$ even}\\
 0 & \textit{ otherwise.}
\end{array}
\right.
 $$
Thus
$$ Z(R) \cap G(R) = \langle K^{r} \rangle, $$
so we can choose
$$ z := K^{r m}, $$ for some $m \in \mathbb{Z}.$ Then we define
$$ y_+ = K^{r s} \quad \textit{ and } \quad y_- = K^{r p}, $$
for $s,p \in \mathbb{Z}$ with $s + p = m.$ Finally, we set
$$ h = \lambda (1 - z) $$
for some $\lambda \in k$ which after normalisation we can take to be either 0 or 1.

Summarising the above calculations, we have proved:

\begin{prop} Let $R$ be the quantised enveloping algebra $U_q (\mathfrak{sl}_2)$, where $q \in k$, $q \neq 0, \pm 1,$ as presented at
the start of (\ref{qsl}). The following is a complete list of the ambiskew Hopf algebras $A = \langle R, X_+, X_- \rangle$ satisfying
(\ref{deltaxpm1}), after a change of variables so that (\ref{hatform}) holds.

\begin{enumerate} \item Let $\chi = \varepsilon$, the counit, with $q$ arbitrary. Set $$A = R[X_+,X_-] = R \otimes k[X_=,X_-],$$ with
$X _+, X_- \in P(A).$
\item Let $\chi = \varepsilon$, the counit, with $q$ a primitive $\ell^{th}$ root of 1. Set $r$ to be $\ell$ if $\ell$ is odd and
$\ell/2$ if $\ell$ is even. Fix integers $m,s$ and $p$ with $m = s + p,$ and $\lambda \in \{0,1\}.$ Define $A = \langle R,
X_+,X_- \rangle,$, where $X_+$ and $X_-$ commute with $R$,
$$ X_+ X_- - X_- X_+ = \lambda(1 - K^{rm}), $$
and $$X_+ \in P_{1,K^{rs}}(A), \qquad X_- \in P_{1,K^{rp}}(A). $$
\item Let $\chi$ be the non-trivial character of $R$, with $q$ a primitive $4\widehat{\ell}^{th}$ root of 1. Choose integers $m,s$
and  $p$, with $s$ and $p$ odd and $s + p = 2m,$ and $\lambda \in \{0,1\}.$ Set $\xi := (-1)^{\widehat{\ell}}.$ Define $A =
\langle R, X_+,X_- \rangle,$, where$$ X_{\pm}E = -EX_{\pm}, \quad X_{\pm}F =  FX_{\pm}, \quad X_{\pm}K = -KX_{\pm},$$ with
$$ X_+ X_- - (-1)^{\widehat{\ell}}X_- X_+ = \lambda (1 - K^{2\widehat{\ell}m}),$$
where $$X_+ \in P_{1,K^{\widehat{\ell}s}}(A), \qquad X_- \in P_{1,K^{\widehat{\ell}p}}(A). $$
\end{enumerate}
\end{prop}

Note that in part (iii) of the proposition, $y_{\pm}$ are not in the centre of $R$, in view of (\ref{centre1}). Thus the statement in
the Main Theorem that $y_{\pm} \in Z(G(R))$ cannot in general be improved.

\section{Coradical filtration}
\subsection{Hypotheses, notation and definitions} \label{coradhyp} The \emph{coradical filtration} $\{C_n : n \geq 0 \}$ of a
coalgebra $C$ is defined by taking $C_0$ to be the
\emph{coradical} of $C$, that is the sum of the simple
subcoalgebras of $C$, and setting, for all $n \geq 0$,
\[ C_{n+1} := \{ c \in C \colon \Delta(c) \in C_{n} \otimes C + C
\otimes C_0\}.\] The apparent left-right asymmetry in this
definition is not genuine, as can be easily deduced from
``associativity of the wedge'', \cite[(5.2.5)]{Mo}. Recall that a
\emph{coalgebra filtration} of a coalgebra $C$ is a family $\{ F_n
: n \geq 0 \}$ of subspaces of $C$ such that (i) $F_n \subseteq
F_{n+1}$, (ii) $C = \cup_{n \geq 0} F_n$ and (iii) $\Delta (F_n)
\subseteq \sum_{i=0}^n F_i \otimes F_{n-i}.$ We record below some
key facts we shall need; other basic properties can be found in
\cite[Section 5.2]{Mo}, for example.

\begin{lemma} \label{filtco} Let $C$ be a coalgebra with coradical filtration
$\{C_n : n \geq 0 \}.$
\begin{enumerate}
\item $\{C_n \}$ is a coalgebra filtration of $C$. \item If $\{
F_n :n \geq 0 \}$ is any coalgebra filtration of $C$ with $F_0
\subseteq C_0$, then $F_n \subseteq C_n$ for all $n \geq 0.$ \item
If $D$ is any subcoalgebra of $C$, then the $n$th term of the
coradical filtration of $D$ is $D \cap C_n$. \item Let $\{ A_i : i
\in \mathcal{I} \}$ be subcoalgebras of $C$ with $\sum_{i \in
\mathcal{I}} A_i = C.$ For each $i \in \mathcal{I}$ and each $n
\geq 0$ let $A_{i,n}$ be the $n$th term of the coradical
filtration of $A_i$. Then, for all $n \geq 0$ and for all $i \in
\mathcal{I},$
$$  C_n = \sum_{i \in \mathcal{I}} A_{i,n}. $$
\end{enumerate}
\end{lemma}
\begin{proof} (i) This is \cite[Theorem 5.2.2]{Mo}.

(ii) This is a straightforward induction on $n$.

(iii) This is \cite[Lemma 5.2.12]{Mo}.

(iv) This follows from (iii) and by an induction on $n$, using the
fact \cite[Proposition 5.2.9(2)]{Mo} that $C_n$ is the $n$th term
in the socle series of $C$ as a left $C^{\ast}-$module. \qedhere
\end{proof}

We shall assume throughout $\S5$ that $R$ is a Hopf $k$-algebra,
 and $A = \langle R,
X_{\pm} \rangle$ is an ambiskew algebra such that the Hopf
structure on $R$ extends to $A$, so that (\ref{deltaxpm1}) holds.
In fact, as shown in Lemma \ref{relabel}, we may and will assume
throughout that the variables have been changed if necessary so
that \eqref{hatform} holds. The notation for the constituents of
$A$ will be as in the Main Theorem, except that we write $X_{\pm}$
for the ambiskew variables. Let $\{ R_a \colon a \geq 0\}$ denote
the coradical filtration of $R,$ $\{ A_t \colon t \geq 0\}$ the
coradical filtration of $A$.

We shall further assume throughout the rest of $\S 5$ that
\begin{equation}\label{char} k \textit{ has characteristic 0.}
\end{equation}
There is no doubt that this additional hypothesis could be removed
by refinements of the arguments below, along the lines of
\cite{B} and \cite{Mu}.

In the above set-up we have in particular that $\xi$ is a non-zero
element of $k$. When $\xi$ is a primitive $d$th root of unity for
some $d
> 1$, define, for each non-negative integer $m$, non-negative integers $q_m$ and $r_m$ by
\[ m := dq_m + r_m,\] with $0 \leq r_m < d.$ If $\xi = 1$ or $\xi$ is not a root of unity, set $q_m := m$ and $r_m := 0$.

Define a partial order $\prec$ on $\mathbb{Z}_{\geq 0}$; this is
dependent on $d$, but we suppress that in the notation. Set $p
\prec m$ if $q_p \leq q_m$ and $r_p \leq r_m$, for any $p,m \in
\mathbb{Z}_{\geq 0}$. Thus if $\xi = 1$ or $\xi$ is not a
non-trivial root of unity, then $p \prec m$ if and only if $p \leq
m$, for all non-negative integers $m$ and $p$. Set $\widehat{m} :=
q_m + r_m$. It is easy to check that, for all non-negative
integers $m, p$ with $p \prec m$,
\begin{equation} \label{precminus}
\widehat{m-p} = \widehat{m} - \widehat{p}.
\end{equation}

Let $0,1 \neq q \in k$. For any integer $n > 0$, let
\[ (n)_q := 1 + \ldots + q^{n-1} = \frac{q^n - 1}{q-1}.\]
The \emph{$q$-factorial} is defined by setting $(0)!_q := 1$ and
\[ (n)!_q := (1)_q (2)_q \cdots (n)_q = \frac{(q-1)(q^2-1) \cdots
(q^n-1)}{(q-1)^n}.\] Thus $(n)!_q$ is a polynomial in $q$ with
integer coefficients. For $0 \leq i \leq n$, we define the
\emph{Gaussian binomial coefficient} \index{Gaussian binomial
coefficient} or \emph{$q$-binomial coefficient} by
\[ \binom{n}{i}_q = \frac{(n)!_q}{(i)!_q (n-i)!_q}.\]
This is also a polynomial in $q$ with integer coefficients \cite[Proposition IV.2.1(a)]{Kas95}. There are analogues of the identities
for binomial coefficients \cite[Proposition IV.2.1(c)]{Kas95}:
\begin{equation} \label{qbinom} \binom{n}{i}_q = \binom{n-i}{i}_q +
q^{n-i} \binom{n-1}{i-1}_q = \binom{n-1}{i-1}_q + q^i
\binom{n-1}{i}_q.\end{equation}  If $q$ is a primitive $n$th root
of unity, then \begin{equation} \label{qbinomun} \binom{n}{i}_q =
0, \enspace \textrm{for all} \enspace 1 \leq i \leq n-1.
\end{equation}

\subsection{Comultiplication formulas} The $q$-binomial theorem \cite[I.6.1]{BG2} at once yields (i) of the following; (ii) follows
easily.
\begin{lemma} \label{coform} Keep the hypotheses and notation of (\ref{coradhyp}),
 and let $m$ and $n$ be non-negative integers.
 \begin{enumerate} \item \label{coform1}
\[ \Delta ({X_{\pm}}^m) = \sum_{j=0}^m
\binom{m}{j}_{{\xi}^{\pm 1}} {y_{\pm}}^{m-j}{X_{\pm}}^j \otimes
{X_{\pm}}^{m-j}.\] \item \label{coform2}
$$ \Delta({X_+}^m{X_-}^n) = \sum_{j=0}^m \sum_{k=0}^n
\binom{m}{j}_{\xi} \binom{n}{k}_{{\xi}^{-1}} \xi^{j(n-k)}
{y_+}^{m-j}{y_-}^{n-k} {X_+}^j{X_-}^k \otimes
{X_+}^{m-j}{X_-}^{n-k}.$$
 \item \label{newcoform1} There are non-zero scalars $\alpha_{p} \in k,$ for all $p \prec
 m,$ such that
 \[\Delta({X_{\pm}}^m) =
\sum_{0 \leq p \prec m}
  \alpha_{p} {y_{\pm}}^{m - p}{X_{\pm}}^{p} \otimes
{X_{\pm}}^{m - p}.\]
 \item There are non-zero scalars $\beta_{v,p} \in k,$ for all $v \prec
 m$ and $p \prec n,$ such that \[ \Delta({X_+}^m{X_-}^n) = \sum_{v \prec m} \sum_{p \prec n}
 \beta_{v,p} {y_{+}}^{m - v} {y_{-}}^{n - p}{X_{+}}^{v}{X_{-}}^{p} \otimes
 {X_{+}}^{m - v}{X_{-}}^{n - p}.\]  \label{newcoform2} \end{enumerate}  \end{lemma}

\begin{proof} (iii) When $\xi = 1$ or $\xi$ is not a non-trivial root of unity, this is just (i). Suppose now that $\xi$ is a primitive $d$th root of unity for $d >
1.$ Using \eqref{qbinomun}, (i) implies \[ \Delta({X_{\pm}}^d) =
{y_{\pm}}^{d} \otimes {X_{\pm}}^d + {X_{\pm}}^d \otimes 1.\]
Furthermore, by Theorem \ref{newhopf}(i), $\sigma^{\pm
1}({y_{\pm}}^d) = \xi^{\pm d}{y_{\pm}}^d = {y_{\pm}}^d$, and so
${y_{\pm}}^d$ commutes with $X_{\pm}$. Therefore, for all integers
$a \geq 0$,
\[ \Delta({X_{\pm}}^d)^a =  \sum_{i=0}^a \binom{a}{i} {y_{\pm}}^{d(a-i)}{X_{\pm}}^{di} \otimes
{X_{\pm}}^{d(a-i)},\] where $\binom{a}{i}$ denotes the ordinary binomial coefficient. Similarly, $X_{\pm}^d$ commutes with $y_{\pm}.$
Setting $m := dq_m + r_m$ and using the above equation together with (i) gives
\begin{align*} \Delta({X_{\pm}}^m) &= \Delta({X_{\pm}}^{d})^{q_m}\Delta({X_{\pm}}^{r_m})\\
&= \left( \sum_{i=0}^{q_m} \binom{q_m}{i}
{y_{\pm}}^{d(q_m-i)}{X_{\pm}}^{di} \otimes {X_{\pm}}^{d(q_m-i)}
\right)\\ & \qquad \qquad \cdot \left( \sum_{j=0}^{r_m}
\binom{{r_m}}{j}_{{\xi}^{\pm 1}} {y_{\pm}}^{{r_m}-j}{X_{\pm}}^j
\otimes {X_{\pm}}^{{r_m}-j} \right)\\ &= \sum_{i=0}^{q_m}
\sum_{j=0}^{r_m}
 \binom{q_m}{i} \binom{{r_m}}{j}_{{\xi}^{\pm 1}} {y_{\pm}}^{d(q_m-i) +{r_m}-j}{X_{\pm}}^{di + j} \otimes
{X_{\pm}}^{d(q_m-i) + {r_m}-j} \\  &= \sum_{i=0}^{q_m}
\sum_{j=0}^{r_m}
 \binom{q_m}{i} \binom{{r_m}}{j}_{{\xi}^{\pm 1}} {y_{\pm}}^{m - (di +j)}{X_{\pm}}^{di + j} \otimes
{X_{\pm}}^{m - (di+j)}\\ &= \sum_{0 \leq p \prec m} \alpha_{p}
{y_{\pm}}^{m - p}{X_{\pm}}^{p} \otimes {X_{\pm}}^{m - p},
\end{align*} where $\alpha_p := \binom{q_m}{i}
\binom{{r_m}}{j}_{{\xi}^{\pm 1}}$ for $p := di + j$, with $0 \leq i \leq q_m$ and $0 \leq j \leq r_m \leq d-1$. Our hypotheses now
ensure that $\alpha_{p} \neq 0$, for all $p \prec m$.

(iv) This follows easily from (iii). \qedhere
\end{proof}

\subsection{The filtration}
Define subspaces $\{F_t \colon t \geq 0\}$ of $A$ by
\[ F_t := \sum_{q + \widehat{m} + \widehat{n} \leq t} R_q
{X_{+}}^m{X_-}^n.\]
\begin{lemma} \label{fcof} Keep the hypotheses and notation of (\ref{coradhyp}),
 and let $t,m,n,\ell,v$ be non-negative integers.
 \begin{enumerate} \item $\{F_t : t \geq 0 \}$ as defined above is a coalgebra
filtration of $A$. \item If $v \prec \ell \prec m$, then $\ell - v
\prec m.$ \item Suppose that $\widehat{m} + \widehat{n} \leq t$,
and define $a := t - \widehat{m} - \widehat{n}.$ Then
$$ F_{t,(m,n)} := \sum_{\ell \prec m, w \prec n} R_a X_+ ^{\ell}
X_- ^{w} $$ is a subcoalgebra of $F_t$.
\end{enumerate}
\end{lemma}

\begin{proof} (i) It is clear that $F_t \subseteq F_{t+1}$ for all $t \geq 0$
and that $A = \bigcup_{t \geq 0} F_t$, so it remains to prove that
$\Delta(F_t) \subseteq \sum_{i=0}^t F_i \otimes F_{t-i}$. Let
$r{X_+}^m {X_-}^n \in F_t$, so $r \in R_a$, with $a + \widehat{m}
+ \widehat{n} \leq t$. Lemma \ref{coform}(iv) shows that
$\Delta(r{X_+}^n{X_-}^m)$ is a sum of terms of the form
 \begin{align} \label{precdel}
\beta_{v,p} r_1{y_{+}}^{m - v} {y_{-}}^{n -
p}{X_{+}}^{v}{X_{-}}^{p} \otimes r_2{X_{+}}^{m - v}{X_{-}}^{n -
p},\end{align} where $v \prec n$, $p \prec m$, $\beta_{v,p} \in
k^*$. By definition of the coradical filtration of $R$, we have
$\Delta(r) \in \sum_{j=0}^a R_j \otimes R_{a-j}$, so without loss
of generality we can suppose that $r_1 \in R_j$ and $r_2 \in
R_{a-j}$ for some $j$, $0 \leq j \leq a$. Since $y_{\pm} \in
G(R)$, $r_1{y_{+}}^{m - v} {y_{-}}^{n - p} \in R_j$. Therefore,
\eqref{precdel} is contained in $F_u \otimes F_w$, where $u = j +
\widehat{v} + \widehat{p}$ and $w = a-j + \widehat{m-v} +
\widehat{n-p}$. It is enough to show that $u + w \leq t$. We have,
using \eqref{precminus},
$$ u + w = a + \widehat{v} + \widehat{p} + \widehat{m-v} +
\widehat{n-p} = a + \widehat{m} + \widehat{n} \leq t,$$ as
required.

(ii) This follows routinely from the definition of $\prec$.

(iii) As in the proof of (ii), if $r \in R_a$, $\ell \prec m$ and
$w \prec n$, so that $rX_+ ^{\ell}X_- ^w \in F_{t,(m,n)}$, then
$\Delta (rX_+ ^{\ell}X_- ^w )$ is a sum of terms of the form
\begin{equation} \label{gulp} \beta_{v,p}r_1 y_+ ^{\ell - v} y_-
^{w-p}X_+ ^v X_- ^p \otimes r_2 X_+ ^{\ell - v}X_- ^{w-p},
\end{equation} where $v \prec \ell, p \prec w, \beta_{v,p} \in
k^*, r_1 \in R_j, r_2 \in R_{a-j},$ for some $j$, $0 \leq j \leq
a.$ Since, in view of (ii), both tensorands in \eqref{gulp} are in
$F_{t,(m,n)}$, the result follows.
\end{proof}

\begin{thrm} \label{coradthrm} Keep the hypotheses and notation of (\ref{coradhyp}). Denote the coradical filtration of $A$ by $\{A_t : t \geq 0 \}.$
Then
\[ A_t = \sum_{q + \widehat{m} + \widehat{n} \leq t} R_q {X_+}^m{X_-}^n.\]
\end{thrm}

\begin{proof} As before, set $F_t := \sum_{a + \widehat{m} + \widehat{n}} R_a X_+ ^m X_- ^n$.
We must show that $F_t = A_t$ for all $t \geq 0$.

Since $F_0 = R_0 = A_0,$ Lemmas \ref{fcof}(i) and \ref{filtco}(ii)
show that $F_t \subseteq A_t$ for all $t \geq 0.$

Now we show by induction on $t$ that $F_t = A_t$ for all $t \geq
0$. Suppose that $t \geq 0,$ that we have proved that $F_i = A_i$
for all $i \leq t$, but that \begin{equation} \label{false}
F_{t+1} \subsetneq A_{t + 1}.\end{equation} Since $\cup_{i \geq 0}
F_i = A$, there exists $t_0 \geq t + 2,$ $t_0$ minimal such that
\begin{equation} F_{t+1} \subsetneq F_{t_0} \cap A_{t+1}.
\label{alpha} \end{equation} For each ordered pair $(m,n)$ of
non-negative integers with $\widehat{m} + \widehat{n} \leq t_0,$
write $E_{t_0, (m,n)}$ for $F_{t_0,(m,n)}+F_{t_0 - 1}$. Observe
that, by choice of $t_0$,
\begin{equation} \label{beta} F_{t_0 - 1} \cap A_{t+1} = F_{t+1}.
\end{equation}
Each $E_{t_0,(m,n)}$ is a subcoalgebra of $F_{t_0}$, by Lemma
\ref{fcof}(iii), and
$$ F_{t_0} = \sum_{m,n}E_{t_0,(m,n)}.$$ By Lemma \ref{filtco}(iii), the $(t+1)$st term of the
coradical filtration of $E_{t_0,(m,n)}$ is $E_{t_0,(m,n)} \cap
A_{t+1}$, which contains $F_{t+1}$. By Lemma \ref{filtco}(iv) and
\eqref{alpha}, there exists at least one ordered pair $(m,n)$ with
\begin{align*} F_{t+1} \quad &\subsetneq \quad A_{t+1} \cap
E_{t_0,(m,n)} \\ &= \quad A_{t+1} \cap (F_{t_0, (m,n)} + F_{t_0 - 1}) \\
&= \quad (A_{t+1} \cap F_{t_0,(m,n)}) + (A_{t+1} \cap F_{t_0 - 1})
\end{align*} The second equality above follows from Lemmas
\ref{fcof}(iv) and \ref{filtco}(iv). Using \eqref{beta}, we deduce
that
\begin{equation} \label{delta} F_{t+1}  \subsetneq (A_{t+1} \cap F_{t_0,(m,n)}) +
F_{t+1} \subseteq E_{t_0,(m,n)} .\end{equation} Define the
non-negative integer $a$ by
\begin{equation} \label{eps} a + \widehat{m} + \widehat{n} = t_0.
\end{equation} Now \eqref{delta} shows that there exist $w, w' \in
A$ and $r \in R_a \setminus R_{a-1}$ such that
$$ w \quad = \quad rX_+ ^m X_- ^n + w', $$ where $w'$ is a sum of
terms $r'X_+ ^i X_- ^j$ where $r'$ has coradical degree $a'$ with
$$ a' + \widehat{i} +\widehat{j} \leq t_0 - 1. $$ By definition of
the coradical filtration and by the induction hypothesis,
$$ \Delta (w) \in A_t \otimes A + A \otimes A_0 = F_t \otimes A +
A \otimes F_0. $$ Now $\Delta (w)$ has the form
\begin{equation} \label{deltasump} \Delta (r) \sum_{v \prec m} \sum_{p \prec n}
\beta_{v,p} {y_{+}}^{m - v} {y_{-}}^{n - p}{X_{+}}^{v}{X_{-}}^{p}
\otimes {X_{+}}^{m - v}{X_{-}}^{n - p}, \end{equation} and we can
find a term $x \otimes z$ in the expansion of $\Delta (r)$ with $x
\in R_a \setminus R_{a-1},$ $z \neq 0.$ Consider now the
corresponding term \eqref{crime} in the expansion of $\Delta(w),$
noting first that it cannot cancel with any terms from the
expansion of $\Delta (w')$:
\begin{equation} \label{crime}
 \beta_{v,p} x{y_{+}}^{m - v} {y_{-}}^{n - p}{X_{+}}^{v}{X_{-}}^{p} \otimes
 z{X_{+}}^{m - v}{X_{-}}^{n - p} \end{equation} A summand of
 \eqref{deltasump} is in $A \otimes F_0$ if and only if $v = m$
 and $p = n;$ all other terms must be in $F_t \otimes A$. So, by
 definition of $F_t$, we have
 \begin{equation} a + \widehat{v} + \widehat{p} \leq t \label{eta}
 \end{equation} for all $v \prec m$ and $p \prec n$ except $(v,p)
 = (m,n)$. As in (\ref{coradhyp}), write $m = q_m d + r_m$ and $n =
 q_n d + r_n$ with $q_m,q_n \in \mathbb{Z}_{\geq 0},$ $0 \leq r_m,r_n \leq
 d-1.$ It is easy to check that we cannot have $m = n = 0;$ so we
 assume without loss of generality that $m \geq 1.$ There are now
 five cases to consider.

 (i) Suppose that $\xi = 1$ or $\xi$ is not a non-trivial root of
 unity. Thus $\prec$ is $\leq$. Consider the term of form
 \eqref{crime} given by $v = m-1,$ $p = n,$ which is in $F_t
 \otimes A;$ then \eqref{eta} yields
 $$ a + m - 1 + n \leq t, $$ which by \eqref{eps} implies
 $$ t_0 - 1 \leq t, $$
 contradicting $t_0 \geq t + 2.$

In the remaining four cases, $\xi$ is a primitive $d$th root of
unity for some $d
> 1$.

(ii) Suppose that $r_m = r_n = 0$. Therefore, $m = q_m d$ and $n =
q_n d$ for some $q_m \geq 1$ and $q_n \geq 0$. Consider the term
of form \eqref{crime} obtained by taking $v := (q_m -1)d$ and $p
:= n$, which is contained in $F_t \otimes A$. Then $\widehat{v} =
\widehat{m}-1$, so \eqref{eta} yields $$ a + \widehat{m} - 1 +
\widehat{n}  \leq t,$$ yielding thanks to \eqref{eps}
$$ t_0 - 1 \leq t,$$
again a contradiction.

(iii) Suppose that $r_m \geq 1$ and $r_n =0$. Consider the summand
of \eqref{deltasump} given by taking $v := m-1$ and $p := n$,
which is contained in $F_t \otimes A$. Then $\widehat{v} =
\widehat{m}-1$ and \eqref{eta} gives $a + \widehat{m} -1 +
\widehat{n} \leq t$, a contradiction.

Cases (iv) $r_m = 0, \, r_n \geq 1,$ and (v) $r_m, r_n \geq 1$ are
similar to (iii).

Thus in all cases the hypothesis \eqref{false} produces a
contradiction. So $F_{t+1} = A_{t+1}$, completing the proof of the
induction step, and with it, the theorem.
\end{proof}

For convenience we state the following immediate consequences of
the theorem:

\begin{cor} \label{pointedcor} Assume the hypotheses of Theorem
\ref{coradthrm}. \begin{enumerate} \item \quad $A_0 = R_0.$ In
particular, $A$ is pointed if and only if $R$ is pointed. \item \[
A_1 = \left\{
\begin{array}{cc} R_1 \oplus R_0{X_{\pm}}, & \xi = 1 \enspace
\textrm{or} \enspace \xi \enspace \textrm{not a root of 1};\\ R_1
\oplus R_0{X_{\pm}} \oplus R_0{X_{\pm}}^{d}, & \xi \enspace
\textrm{a primitive} \enspace d\textrm{th root of 1.}
\end{array} \right. \]
\end{enumerate}
\end{cor}

\section{Properties of ambiskew Hopf algebras}
\subsection{Ring-theoretic properties}\label{ring}
The additional terminology used throughout $\S 6$ is standard and can be
checked as required in the references provided. However, to state
in Theorem \ref{ring}(vi) the conditions required for $A$ to
satisfy a polynomial identity we need some further notation, as
follows. Suppose that $k$ is algebraically closed of
characteristic 0, and that the automorphism $\sigma$ of $R$ has
finite order $n$. Let $\eta$ be a primitive $n$th root of unity in
$k$. Then
$$ R = \bigoplus_{i = 0}^{n-1} R_i, $$
where $R_i = \{ r \in R : \sigma (r) = \eta^i r \}.$ In particular, \begin{equation} \label{hdec} h = \sum_{i=0}^{n-1} h_i, \end{equation}
with $h_i \in R_i$ for all $i$.

\begin{thrm} Let $R$ be a $k$-algebra and let $A =  A(R, X_+, X_-, \sigma,
h, \xi)$ be an ambiskew algebra over $R$ as in \ref{ambiskew}.
 \begin{enumerate}
\item $A$ is noetherian if and only if $R$ is noetherian.
 \item $A$ is a domain if and only if $R$ is a domain.
 \item $A$ is semiprime left Goldie if and only if $R$ is semiprime left Goldie.
 \item If $R$ is prime then $A$ is prime.
 \item Suppose that every finite subset of $R$ is contained in a finite dimensional $\sigma-$invariant subspace of $R$. Then $A$ has finite Gel'fand-Kirillov dimension
if and only if $R$ does also. In this case
\begin{equation}\label{gkdim}  \mathrm{GK-dim}(A) = \mathrm{GK-dim}(R) + 2. \end{equation}
 \item Suppose that $k$ is algebraically closed of characteristic 0 and that $R$ is a commutative affine domain. Then the following conditions are equivalent:
\begin{enumerate} \item $A$ satisfies a polynomial identity.
\item $A$ is a finite module over its centre.
\item The following
conditions hold:
\begin{itemize} \item $\sigma_{|R}$ has finite
order $n$;
 \item $\xi$ has finite order $t$;
  \item if $t | n$, so
that $\xi = \eta^j$ for some $j$, $0 \leq j \leq n-1,$ then $h_j =
0$ in \eqref{hdec}.
\end{itemize}
\end{enumerate}

When the above equivalent conditions hold, the P.I. degree of $A$ is $2mn$, where $m = l.c.m.(n,t).$

\end{enumerate}
\end{thrm}

\begin{proof} (i),(ii) \cite[Theorem 1.2.9(iv),(i)]{MR}.

(iii) \cite[Theorem 2.6]{Ma}.

(iv) \cite[Theorem 1.2.9(iii)]{MR}.

(v) Using the iterated skew structure of $A$ as
 explained in \ref{iterate}, set $B:= R[X_+; \sigma]$. Then the result
  follows by two applications of \cite[Lemma 4.1]{Z}. For
the second application, that is to pass from $B$ to $A$, one has
to check first that $\sigma$ is locally finite dimensional when
extended to $B$, and then that every finite set of elements of $B$
is contained in a finite dimensional subspace $V$ for which the
subalgebra $k\langle V \rangle$ of $B$ is $\delta - $stable. The
first statement is true because $\sigma$ is extended to $B$ by
setting $\sigma (X_+) = \xi X_+$, as discussed in \ref{iterate},
and the second is clear since $\delta (R) = 0$ and $\delta$ is
locally nilpotent on $B$.

(vi) This is proved in \cite[$\S$ 5.4.2]{M}. \end{proof}

\subsection{Remarks:}(i) The converse of (iv) is false, as shown by taking $R = k \oplus k$ and $\sigma((a,b) = (b,a)$, with (say) $h = 0$ and $\xi = 1$.

(ii) We have omitted the proof of (vi) because it is rather long
and technical, and its inclusion would take us too far away from
the main theme of this paper. The hypothesis there that $k$ has
characteristic 0 is necessary, as is illustrated by the first Weyl
algebra $A_1(k)$, whose centre is $k[X^p, Y^p]$ when $k$ has
characteristic $p > 0.$ In this case $h = h_0 = 1 \neq 0$. It is
somewhat surprising to us that the question of when an ambiskew
algebra satisfies a polynomial identity does not seem to have been
previously addressed in the literature.

\subsection{Homological properties} \label{homol} We recall from \cite{BZ} that a noetherian $k$-algebra $T$ with
a fixed augmentation $\varepsilon : T \longrightarrow k$ is \emph{AS-Gorenstein} if the left $T$-module $T$
has finite injective dimension $d$, with $\mathrm{Ext}^i_T(k,T) = \delta_{id} k$, and the same condition
holds on the right. Then $T$ is \emph{AS-regular} if, in addition, $\mathrm{gl.dim.}(T) < \infty.$

\begin{thrm} Let $R$ be a $k$-algebra and let $A =  A(R, X_+, X_-, \sigma,
h, \xi)$ be an ambiskew algebra over $R$ as in \ref{ambiskew}.
 \begin{enumerate}
\item If $R$ has finite global dimension ($\mathrm{gl.dim.}$) then so also does $A$, and in this case
\begin{equation}\label{gldim} \mathrm{gl.dim.}(R) + 1 \leq \mathrm{gl.dim}(A) \leq \mathrm{gl.dim.}(R) + 2. \end{equation}
\item If $R$ has finite injective dimension ($\mathrm{inj.dim.}$) then so does $A$, and in this case
\begin{equation}\label{injdim} \mathrm{inj.dim.}(R) + 1 \leq \mathrm{inj.dim}(A) \leq \mathrm{inj.dim.}(R) + 2. \end{equation}
\item If $R$ is Auslander-Gorenstein then so is $A$.
\item If $R$ is Auslander-regular then so is $A$.

 \end{enumerate}
  \end{thrm}

 \begin{proof}  (i) With the notation introduced in the proof of Theorem \ref{ring}(v), $\mathrm{gl.dim.}(B) = \mathrm{gl.dim.}(R) + 1$ by \cite[Theorem 7.5.3(iii)]{MR}. Now $A = B[X_-; \sigma^{-1}, \delta]$, so that $\mathrm{gl.dim.}(B) \leq \mathrm{gl.dim.}(A) \leq \mathrm{gl.dim.}(B) + 1$ by \cite[Theorem 7.5.3(i)]{MR}. Combining these yields the desired inequalities.

(ii), (iii) These are similar to (i), using \cite[Theorem 4.2]{E}.

(iv) Immediate from (i) and (iii). \end{proof}

\subsection{Remark:} Both inequalities in  \eqref{gldim} (and hence in \eqref{injdim}) can be attained.
Take $R = k$ and $A = k[X,Y]$ to get equality at the higher end,
and $R = k$ (of characteristic 0) and $A = A_1(k)$, the first Weyl
algebra, to get equality at the lower end.

\subsection{Hopf algebraic properties}\label{hoppy} Part (ii) of the result below should be read in the context
of the (currently open) question \cite[1.15]{BG1}, \cite[Question E]{Br} whether every
noetherian Hopf algebra is AS-Gorenstein.

\begin{thrm} Let $R$ be a $k$-algebra and let $A =  A(R, X_+, X_-, \sigma,
h, \xi)$ be an ambiskew algebra over $R$ as in \ref{ambiskew}. Suppose that $A$ admits a Hopf algebra structure extending a Hopf structure on $R$, and such that \eqref{deltaxpm1} holds.
\begin{enumerate}
\item   The second inequality in \eqref{gldim} is an equality.
\item $A$ is AS-Gorenstein if and only if $R$ is AS-Gorenstein. In this case, the second inequality in \eqref{injdim} is an equality.
\item $A$ is AS-regular if and only if $R$ is AS-regular.
\item $A$ has finite Gel'fand-Kirillov dimension
if and only if $R$ does also. In this case \eqref{gkdim} holds.
\end{enumerate}
\end{thrm}

\begin{proof}(i) By \cite{LL}, the global dimension of a Hopf $k$-algebra $H$ is the projective
dimension of its trivial module $k$. It then follows by an easy application of the long exact sequence of cohomology that, when finite, \begin{equation} \label{triv} \mathrm{gl.dim.}(H) = \mathrm{max}\{ i : \mathrm{Ext}^i_H(k,H) \neq 0 \}. \end{equation} Since $k$ is manifestly a finitely generated module over $B$ and $R$, two applications of \cite[Theorem 8]{S} yield
\begin{equation}\label{ASjump} \mathrm{Ext}^{i+2}_A(k, A) \cong \mathrm{Ext}^{i}_R(k,R) \end{equation}
for all $i \geq 0$. (Note that there are typos in the statement of this theorem in \cite{S}.)
The result follows from \eqref{triv} and \eqref{ASjump}.

(ii) Immediate from \eqref{injdim} and \eqref{ASjump}.

(iii) Immediate from \eqref{gldim} and \eqref{ASjump}.

(iv) We have to show that the hypothesis on $\sigma$ needed for
Theorem \ref{ring}(v) is satisfied. This follows from the fact,
proved in the Main Theorem \ref{newhopf}, that $\sigma$ is a
winding automorphism, combined with the basic property of
coalgebras that every finite set of elements of a coalgebra is
contained in a finite dimensional subcoalgebra, \cite[Theorem 5.1.1]{Mo}.
\end{proof}

\section*{Acknowledgments} Some of the research described in this paper
formed part of the second author's PhD thesis at the University of
Glasgow \cite{M}, funded by the EPSRC.

\providecommand{\bysame}{\leavevmode\hbox
to3em{\hrulefill}\thinspace}
\providecommand{\MR}{\relax\ifhmode\unskip\space\fi MR }
\providecommand{\MRhref}[2]{%

\href{http://www.ams.org/mathscinet-getitem?mr=#1}{#2} }
\providecommand{\href}[2]{#2}

\end{document}